\documentclass[11pt]{article}

\usepackage[english]{babel}
\usepackage{amssymb}
\usepackage{amsmath}
\usepackage{hyperref}
\usepackage{enumerate}
\usepackage{amsthm}
\usepackage{amsfonts}
\usepackage{listings}
\usepackage[curve]{xypic}

\usepackage{MnSymbol}
\usepackage{stmaryrd}

\DeclareMathSymbol{\mlq}{\mathord}{operators}{``}
\DeclareMathSymbol{\mrq}{\mathord}{operators}{`'}

\title{Descent for coherent sheaves along formal/open coverings}
\date{January 13, 2016}
\author{Fritz H\"ormann\\ Mathematisches Institut, Albert-Ludwigs-Universit\"at Freiburg}

\usepackage[numbers,sort&compress]{natbib}

\usepackage{color}
\usepackage[headsepline,footsepline]{scrpage2}

\usepackage{tikz}

\newtheorem{SATZ}{Theorem}[section]
\newtheorem{HAUPTSATZ}[SATZ]{Main Theorem}
\newtheorem{LEMMA}[SATZ]{Lemma}
\newtheorem{DEF}[SATZ]{Definition}
\newtheorem{PROP}[SATZ]{Proposition}

\newtheorem{BEISPIEL}[SATZ]{Example}

\newtheoremstyle{bare}        
  {}            
  {}            
  {\normalfont}                 
  {}                            
  {\bfseries}                   
  {}                            
  {.0em}                           
  {\thmnumber{#2}#1. \thmnote{\normalfont\textsc{(#3)}} } 

\theoremstyle{bare}
\newtheorem{PAR}[SATZ]{}

\newcommand{\comment}[1]{}

\newcommand{\Z}{ \mathbb{Z} }

\newcommand{\N}{ \mathbb{N} }

\newcommand{\A}{ \mathbb{A} }

\newcommand{\OO}{ {\cal O} }

\DeclareMathOperator{\colim}{colim}

\DeclareMathOperator{\id}{id}

\DeclareMathOperator{\op}{op}

\DeclareMathOperator{\cocart}{cocart}

\DeclareMathOperator{\Hom}{Hom}

\DeclareMathOperator{\End}{End}

\DeclareMathOperator{\Mor}{Mor}

\DeclareMathOperator{\pr}{pr}

\DeclareMathOperator{\Dia}{Dia}

\newcommand{\cat}[1]{ {[\textnormal{ \textbf{#1} }]} }

\begin{document}

\maketitle

{\footnotesize  {\em 2010 Mathematics Subject Classification:} 14C20, 14B20, 18F20, 13J10   }

{\footnotesize  {\em Keywords: coherent sheaves, descent, divisors of strict normal crossings, formal schemes}  }

\section*{Abstract}

For a regular noetherian scheme $X$ with a divisor with strict normal crossings $D$ we prove that coherent sheaves satisfy descent w.r.t.\@ the `covering' consisting of the open parts in the various completions of $X$ along the components of $D$ and their intersections.

\section*{Introduction}

Let $X$ be a regular noetherian scheme and $D \subset X$ a divisor with strict normal crossings (cf.\@ Definition~\ref{DEFDSNC}).
In applications it is useful to have a descent theory for coherent sheaves on $X$ relative to a `covering' constisting of the open parts in the various completions of $X$ along the components of $D$ and their intersections. For example when $X$ is the (integral model of a) toroidal compactification of a Shimura variety, these completions can be described as completions
of relative torus embeddings on mixed Shimura varieties. In the case that $D$ consists of one component, these questions have been treated in the literature, see e.g.\@ \cite{MB96, BL95, Bha14}, \cite[Tag 0BNI]{stacks-project}.
In this short article, we generalize these results to arbitrary divisors with normal crossings, sticking to the noetherian case, however. The result is as follows:   

Let $\{Y\}$ be the coarsest stratification of $X$ into locally closed subschemes such that every component of $D$ is the closure $\overline{Y}$ of a stratum $Y$. For each stratum we define a sheaf of $\OO_X$-algebras on $X$:
\begin{equation}\label{eqintro1}
 R_Y(U) := C_{\overline{Y}} \OO_X(U') =  \lim_n \OO_X(U')/\mathcal{I}_{\overline{Y}}(U')^n   
 \end{equation}
where $U$ is an open subset of $X$, and $U' \subset U$ is such that $U' \cap \overline{Y} = U \cap Y$. The definition of $R_Y$ does not depend on the choice of $U'$ for $R_Y$ may also be described
as the push-forward $\iota_* \OO_{C_{\overline{Y}}X|_Y}$, where $C_{\overline{Y}}X$ is the formal completion of $X$ along $\overline{Y}$, and $C_{\overline{Y}}X|_Y$ is the
open formal subscheme with underlying topological space equal to $Y$ and $\iota$ is the composed morphism of formal schemes $C_{\overline{Y}}X|_Y \rightarrow X$. 

For any chain of disjoint strata $Y_1, Y_2, \dots, Y_n$ such that $Y_i \subset \overline{Y_{i-1}}$ for $i=2, \dots, n$, we define inductively a sheaf $R_{Y_1, \dots, Y_n}$ of $\OO_X$-algebras on $X$
which coincides with (\ref{eqintro1}) for $n=1$. For $n>1$ we set
\begin{equation}\label{constructr}
 R_{Y_1, \dots, Y_n}(U) := C_{\overline{Y_1}} (R_{Y_2, \dots, Y_n}(U) \otimes_{\OO_X(U)} \OO_X(U')).    
 \end{equation}
Here $\otimes$ is the usual tensor product (not completed!) and $U' \subset U$ is again such that $U' \cap \overline{Y_1} = U \cap Y_1$. Again this definition does not depend on $U'$. 

Let 
\[ \cat{$(\int S, R)$-coh-cocart}  \]
be  the category of the following descent data (it will be defined in a different way in the article, which will explain also the notation):
For each stratum $Y$ a coherent sheaf $M_Y$ of $R_Y$-modules together with isomorphisms
\[  \rho_{Y,Z}: M_Y \otimes_{R_Y} R_{Y,Z} \rightarrow M_Z \otimes_{R_Z} R_{Y,Z} \]
for any $Y, Z$ with $Z \subset \overline{Y}$, 
which are compatible w.r.t.\@ any triple $Y, Z, W$ of strata with $Z \subset \overline{Y}$ and $W \subset \overline{Z}$ in the obvious way.

Let
\[ \cat{$\OO_X$-coh} \]
be the category of coherent sheaves on $X$.

Then we have

\vspace{0.2cm}

{\bf Main Theorem \ref{MAINTHEOREM}.} {\em The natural functor
\[ \cat{$\OO_X$-coh} \rightarrow \cat{$(\int S, R)$-coh-cocart}  \]
is an equivalence of categories. }

\section{Generalities}

\begin{PAR}
Let $\mathcal{D} \rightarrow \mathcal{S}$ be a bifibered category such that the fibers have all limits and all colimits. 
We will be interested mainly in the following two cases 
\[ \cat{mod} \rightarrow \cat{ring} \]
where $\cat{ring}$ is the category of commutative rings with 1 and $\cat{mod}$ is the bifibered category of modules over such rings.
Furthermore let $X$ be a topological space. 
Then we consider
\[ \cat{$X$-mod} \rightarrow \cat{$X$-ring} \]
where $\cat{$X$-ring}$ is the category of ring sheaves on $X$ {\em which are coherent over themselves in the sense of ringed spaces}
and $\cat{$X$-mod}$ is the bifibered category of sheaves of modules over such ring sheaves. 

Back in the general case, for a morphism $f \in \Mor(\mathcal{S})$ we denote by $f_\bullet$ and $f^\bullet$ the corresponding push-forward and pull-back functors. 
By definition of a bifibered category $f_\bullet$ is always left adjoint to $f^\bullet$. 
Those functors are only
defined up to a (unique) natural isomorphism. Whenever we write $f_\bullet$ or $f^\bullet$ we assume that a choice has been made. 
For example, for a morphism $f: S \rightarrow T$ of rings, $f^\bullet$ is the forgetful map that considers
a $T$-module as an $S$-module via $f$ (hence there is a canonical choice in this case) and $f_\bullet$ is its left-adjoint, the functor $- \otimes_S T$ (which is also 
only defined up to a unique natural isomorphism). 
\end{PAR}

\begin{PAR}
Pairs $(I, S)$ consisting of a diagram $I$ (i.e.\@ a small category) and a functor $S: I \rightarrow \mathcal{S}$
form a 2-category $\Dia^{\mathrm{op}}(\mathcal{S})$, called the category of diagrams in $\mathcal{S}$ (cf. also \cite{HOR15}). 
A morphism of $\mathcal{S}$-diagrams $(\alpha, \mu): (I, S) \rightarrow (J, T)$ is a functor $\alpha: I \rightarrow J$ together with a 
natural transformation $\mu: T \alpha \Rightarrow S$. A 2-morphism $(\alpha, \mu) \Rightarrow (\beta, \nu)$ is a 
natural transformation $\kappa: \alpha \Rightarrow \beta$ such that $\nu \circ (T \ast \kappa) = \mu$.
\end{PAR}

\begin{PAR}
For each pair $(I, S)$, we define the category $\mathcal{D}(I, S)$ of $(I, S)$-modules, whose objects are lifts of
the functor $S$
\[ \xymatrix{
&&  \mathcal{D} \ar[d] \\
I \ar[rr]_-S \ar[rru]^-M &&  \mathcal{S}
} \]
to the given bifibered category, and whose morphisms are natural transformations between those. 
An object is called {\bf coCartesian} if all the morphisms $M(\rho)$ for $\rho: i \rightarrow j$ are coCartesian. 
This defines a full subcategory $\mathcal{D}(I, S)^{\cocart}$ of $\mathcal{D}(I, S)$. If $p: I \rightarrow E$ is a functor between small categories, we also define the subcategory of {\bf $E$-coCartesian} objects as those for which the morphisms $M(\rho)$ are coCartesian
for all $\rho$ such that $p(\rho)$ is an identity.
\end{PAR}

\begin{PAR}We need a refinement of the categories defined above. Suppose we are given a full subcategory $\mathcal{D}^{\mathrm{f}}$ of $\mathcal{D}$ whose objects shall be called {\bf finite}.
We assume that $\mathcal{D}^{\mathrm{f}} \rightarrow \mathcal{S}$ is still opfibered (i.e.\@ push-forward preserves finiteness) but not necessarily fibered.
We define the full subcategories $\mathcal{D}(I, S)^{\mathrm{f}}$ and $\mathcal{D}(I, S)^{\mathrm{f}, \cocart}$ requiring point-wise finiteness.  
\end{PAR}

\begin{DEF}
In the example $\cat{mod} \rightarrow \cat{ring}$ an object $M \in \cat{mod}$ over a ring $R$ is finite, if it is a finitely generated $R$-module. 

In the example $\cat{$X$-mod} \rightarrow \cat{$X$-ring}$ an object $M \in \cat{$X$-mod}$ lying over the sheaf of rings $R$ is finite, if it is {\bf coherent} in the sense of ringed spaces, i.e.\@ locally on $X$, we have an exact sequence
\[ R|_U^n \rightarrow R|_U^m \rightarrow M|_U \rightarrow 0 \]
for some $n, m \in \N_0$. 
\end{DEF}

\begin{PAR}
For each morphism $(\alpha, \mu): (I, S) \rightarrow (J, T)$, we have a corresponding pull-back $(\alpha, \mu)^*$ given by 
\[ ((\alpha, \mu)^* M)_i =  \mu(i)_\bullet M_{\alpha(i)} \quad \forall i \in I.  \]
$(\alpha, \mu)^*$ has a right adjoint $(\alpha, \mu)_*$ given by Kan's formula
\[ ((\alpha, \mu)_* M)_j =  \lim_{j \times_{/J} I} T(\nu)^\bullet \mu(i)^\bullet M_{i} \quad \forall j \in J,  \]
where an object in the slice category $j \times_{/J} I$ is denoted by a pair $(i, \nu)$, with $i \in I$ and $\nu: j \rightarrow \alpha(i)$. 
If $(\alpha, \mu): (I, S) \rightarrow (J, T)$ is purely of diagram type, i.e.\@ if $\mu: T \circ \alpha \rightarrow S$ is the identity, then $(\alpha, \mu)^*$ does have also a left adjoint $(\alpha, \mu)_!$ given by
\[ ((\alpha, \mu)_! M)_j = \colim_{I \times_{/J}  j} T(\nu)_\bullet M(i) \quad \forall j \in J,  \]
where an object in the slice category $I \times_{/J} j$ is denoted by a pair $(i, \nu)$, with $i \in I$ and $\nu: \alpha(i) \rightarrow j$. 
\end{PAR}

\begin{BEISPIEL}For the fibered category $\cat{mod} \rightarrow \cat{ring}$ and
for each morphism $(\alpha, \mu): (I, S) \rightarrow (J, T)$, the pull-back is given by 
\[ ((\alpha, \mu)^* M)_i =  M_{\alpha(i)} \otimes_{T_{\alpha(i)}} S_i \quad \forall i \in I,  \]
and
\[ ((\alpha, \mu)_* M)_j =  \lim_{j \times_{/J} I} M_{i} \quad \forall j \in J,  \]
where each $M_{i}$ is considered an $T_j$-module via the composition $T_{j} \rightarrow T_{\alpha(i)} \rightarrow S_i $.
If $\mu: T \circ \alpha \rightarrow S$ is the identity, then
\[ ((\alpha, \mu)_! M)_j = \colim_{I \times_{/J}  j} M(i) \otimes_{S_{\alpha(i)}} S_j \quad \forall j \in J.  \]
\end{BEISPIEL}

\section{Morphisms of (finite) descent}

We keep the notations introduced in the previous section. 

\begin{DEF}
We say that a morphism $(\alpha, \mu)$ in $\Dia^{\mathrm{op}}(\mathcal{S})$  is {\bf of descent} if  
\[ (\alpha, \mu)^*: \mathcal{D}(J, T)^{\cocart} \rightarrow \mathcal{D}(I, S)^{\cocart} \]
is an equivalence of categories. 

We say that a morphism $(\alpha, \mu)$ in $\Dia^{\mathrm{op}}(\mathcal{S})$  is {\bf of finite descent} if  
\[ (\alpha, \mu)^*: \mathcal{D}(J, T)^{\mathrm{f}, \cocart} \rightarrow \mathcal{D}(I, S)^{\mathrm{f}, \cocart} \]
is an equivalence of categories. 
\end{DEF}

That $(\alpha, \mu)$ is of descent 
does not imply that $(\alpha, \mu)_*$ is an inverse to the equivalence $(\alpha, \mu)^*$. This holds if and only if $(\alpha, \mu)_*$ preserves coCartesian objects.
Obviously $(\alpha, \mu)_*$ always preserves coCartesian objects if $J=\{\cdot\}$. 

The following proposition lists some of the basic properties of this formalism. Assertions 1.--4. are very general and hold also in the context of an arbitrary fibered derivator. 
Assertions 5.--6. are specific to the situation of an ordinary fibered category. 
\begin{PROP}\label{PROPDESCENT}
\begin{enumerate}
\item The morphisms of descent  (resp.\@ of finite descent) satisfy the 2-out-of-3 property. 
\item For two morphisms of diagrams in $\mathcal{S}$
\[ \xymatrix{ (I, S) \ar@/^3pt/[r]^{(\alpha, \mu)} & (J,T) \ar@/^3pt/[l]^{(\beta, \nu)} } \]
such that sequences of 2-morphisms $(\alpha, \mu) \circ (\beta, \mu) \Rightarrow \cdots \Leftarrow \id$ and $ (\beta, \mu) \circ (\alpha, \mu) \Rightarrow \cdots \Leftarrow \id$ exist, we have that $(\alpha, \mu)$ and $(\beta, \mu)$ are of descent (resp.\@ of finite descent).

\item Let $(\alpha, \mu): (I, S) \rightarrow (J, T)$ be a morphism of diagrams of rings.
If for every $j \in J$ the morphism $(\alpha_j, \mu_j): (j \times_{/J} I, \pr_2^*S) \rightarrow (j, T_j)$ is of descent (resp.\@ of finite descent) then $(\alpha, \mu)$ is of descent (resp.\@ of finite descent).  

 If $\alpha: I \rightarrow J$ is a Grothendieck fibration then the statement holds with the slice category $j \times_{/J} I$ replaced by the fibered product $j \times_{J} I$. 

 \item Let $\alpha: I \rightarrow J$ be a morphism of diagrams and let $(\alpha, \id): (I, \alpha^* S) \rightarrow (J, S)$ be a morphism of diagrams in $\mathcal{S}$ of pure diagram type. If $(I \times_{/J} j, S_k) \rightarrow (\cdot, S_k)$ is of descent (resp.\@ of finite descent) for all $j, k \in J$ then 
 $(\alpha, \id)$ is of descent (resp.\@ of finite descent).

 \item Let $S \in \mathcal{S}$ be an object and $I$ a diagram. Denote by $(I, S)$ the corresponding constant diagram. There is an equivalence of categories 
\[ \mathcal{D}(I, S)^{\cocart} \rightarrow \mathcal{D}(I[\Mor(I)^{-1}], S)^{\cocart} \]
(resp.\@ decorated with $\mathrm{f}$)
where $I[\Mor(I)^{-1}]$ is the universal groupoid to which $I$ maps. $I[\Mor(I)^{-1}]$ is equivalent to the small category whose set of objects is $\pi_0(I)$ and whose morphism sets are
\[ \Hom(\xi, \xi') = \begin{cases} \pi_1(I, \xi) & \text{if } \xi=\xi', \\ \emptyset & \text{otherwise}. \end{cases} \] 

In particular $(I, S) \rightarrow (\cdot, S)$ is of descent (resp.\@ of finite descent) if $\pi_0(I)=\pi_1(I)=\{\cdot\}$.

 \item Let $\Delta$ be the simplex category, and let $\Delta^\circ$ be the injective simplex category. Let $(\Delta, S_\bullet)$, resp.\@ $(\Delta^{\circ}, S_\bullet)$ be
 a cosimplicial, resp.\@ a cosemisimplicial object in $\mathcal{S}$. Then the category $\mathcal{D}(\Delta, S_\bullet)^{\cocart}$, resp.\@ $\mathcal{D}(\Delta^\circ, S_\bullet)^{\cocart}$, is equivalent to the category of classical descent data, whose objects are an object $M$ in $\mathcal{D}(S_0)$ together with an isomorphism in $\mathcal{D}(S_1)$
 \[ \rho: (\delta_1^0)^* M \rightarrow (\delta_1^1)^* M \]
 such that the following equality of morphisms in $\mathcal{D}(S_2)$ holds true:
 \[ (\delta_2^0)^* \rho \ \circ \ (\delta_2^2)^* \rho = (\delta_2^1)^* \rho.    \]
Here $\delta_k^i$  is the strictly increasing map $\{0, \dots, k-1\} \rightarrow \{0,\dots,k\}$ omitting $i$.  
\end{enumerate}
\end{PROP}
\begin{proof}
1.\@ is clear.

2.\@ follows from the fact that if $\rho: (\alpha, \mu) \rightarrow (\beta, \nu)$ is a 2-morphism between 1-morphisms $(I, S) \rightarrow (J, T)$ and $M \in \mathcal{D}(J,T)$ is a coCartesian object, then
the morphism 
\[ \rho^*: (\alpha, \mu)^* M \rightarrow (\beta, \nu)^* M  \]
is an isomorphism. 

3. We start by showing that both unit and counit 
\[ (\alpha, \mu)^* (\alpha, \mu)_* E \rightarrow E \qquad  F \rightarrow (\alpha, \mu)_* (\alpha, \mu)^* F \]
are isomorphisms when restricted to the subcategories of coCartesian objects. 
Let $j: (\cdot, T_j) \rightarrow (J, T)$ be the embedding. We have to check that 
\begin{equation}\label{eq1}
 j^*  E \rightarrow j^* (\alpha, \mu)_* (\alpha, \mu)^* E 
 \end{equation}
is an isomorphism for all $j$. 

Consider the 2-commutative diagram:
\[ \xymatrix{
(j \times_{/J} I, \iota_j^*S)  \ar[r]^-{\iota_j} \ar[d]^{p_j} \ar@{}[rd]|\Nearrow & (I, S)  \ar[d]^{(\alpha, \mu)} \\
(j,T_j) \ar[r] & (J, T)
} \]

By the explicit point-wise formula for $ (\alpha, \mu)_*$, the morphism (\ref{eq1}) is the same as
\[    j^* E \rightarrow p_{j,*} \iota_j^* (\alpha, \mu)^* E_j. \] 
The morphism induced by the 2-morphism in the diagram $\iota_j^* (\alpha, \mu)^* E \rightarrow  p_j^* j^* E$
is an isomorphism on coCartesian objects by definition. Since the unit $\id \rightarrow p_{j,*} p_j^*$ is an isomorphism by assumption, we are done.

We now show that $(\alpha, \mu)_*$ preserves coCartesian objects.
Let $\rho: j_1 \rightarrow j_2$ be a morphism in $J$. It induces a map of fibers (purely of diagram type) $\overline{\rho}: (j_2 \times_{/J} I, \pr_2^*S) \rightarrow (j_1 \times_{/J} I, \pr_2^*S)$.
We have to show that 
\[ S(\rho)^* ((\alpha, \mu)_*M)_{j_1} \rightarrow ((\alpha, \mu)_*M)_{j_2}  \]
is an isomorphism. This can be checked after pull-back along  $p_2: (j_2 \times_{/J} I, \pr_2^*S) \rightarrow (j_2, T_{j_2})$ because this induces an equivalence of the categories of coCartesian objects by assumption. Since $S(\rho) p_2 = p_1 \overline{\rho}$ we get the morphism
\[ \overline{\rho}^* p_1^* ((\alpha, \mu)_*M)_{j_1} \rightarrow p_2^* ((\alpha, \mu)_*M)_{j_2}  \]
which is the same as
\[ \overline{\rho}^* \iota_1^* M \rightarrow \iota_2^* M.  \]
Since $\iota_2 = \iota_1 \circ \overline{\rho}$, this is an isomorphism.

To see that the counit is an isomorphism on coCartesian objects, we have to see that 
\begin{equation}\label{eq2}
 \iota_j^*  (\alpha, \mu)^* (\alpha, \mu)_* E \rightarrow \iota_j^* E 
 \end{equation}
is an isomorphism for all $j$. 
Since $(\alpha, \mu)_*$ preserves coCartesian objects, this is again the morphism
\[
p_j^* j^* (\alpha, \mu)_* E \rightarrow \iota^*   E 
\]
and hence the morphism induced by the counit 
\[
p_j^*p_{j,*} \iota^* E \rightarrow \iota^*   E.
\]
This is an isomorphism on coCartesian objects by assumption.

Proof of the additional statement: If $\alpha$ is a Grothendieck fibration, we have an adjunction
\[ \xymatrix{ (j \times_{J} I, \iota_j^*S) \ar@/^3pt/[r]^{\iota_j} & (j \times_{/J} I, \pr_2^*S) \ar@/^3pt/[l]^{\kappa_j} } \]
with $\kappa_j \iota_j = \id$ and such that there is a 2-morphism $\iota_j \kappa_j \Rightarrow  \id$.
Hence by 2., these morphisms are of descent (resp. of descent for finitely generated modules). Hence we may replace $j \times_{/J} I$ by $j \times_J I$ in the statement.

4. 
We will show again that both unit and counit 
\[ (\alpha, \mu)_! (\alpha, \mu)^* E \rightarrow E \qquad  F \rightarrow (\alpha, \mu)^* (\alpha, \mu)_! F \]
are isomorphisms when restricted to the subcategories of coCartesian objects. 
Let $j: (\cdot, T_j) \rightarrow (J, T)$ be the embedding. We have to see that 
\begin{equation}\label{eq3}
 j^*  (\alpha, \mu)_! (\alpha, \mu)^* E \rightarrow j^* E 
 \end{equation}
is an isomorphism for all $j$. 

Consider the 2-commutative diagram:
\[ \xymatrix{
(I \times_{/J} j, T_j)  \ar[r]^-\iota \ar[d]^{p_j} \ar@{}[rd]|\Swarrow & (I, S)  \ar[d]^{(\alpha, \mu)} \\
(j,T_j) \ar[r] & (J, T)
} \]

By the explicit point-wise formula for $ (\alpha, \mu)_!$, the morphism (\ref{eq3}) is the same as
\[    p_{j,!} \iota^* (\alpha, \mu)^* E \rightarrow  j^* E_j. \] 
The morphism induced by the 2-morphism in the diagram $\iota^* (\alpha, \mu)^* E \rightarrow  p_j^* j^* E$
is an isomorphism on coCartesian objects by definition. Since the counit $p_{j,!} p_j^* \rightarrow \id$ is an isomorphism by assumption, we are done.

We now show that $(\alpha, \mu)_!$ preserves coCartesian objects.
Let $\rho: j_1 \rightarrow j_2$ be a morphism in $J$. 
We have to show that 
\[ S(\rho)^* ((\alpha, \mu)_!M)_{j_1} \rightarrow ((\alpha, \mu)_!M)_{j_2}  \]
is an isomorphism. 
After inserting the point-wise formula and denoting $p_2': (I \times_{/J} j_1, T_{j_2}) \rightarrow (\cdot, T_{j_2})$, $\iota_2': (I \times_{/J} j_1, T_{j_2}) \rightarrow (I, \alpha^*T)$
we get: 
\begin{equation} \label{eqp2}
p_{2,!}'  (\iota_2')^* M \rightarrow p_{2,!} \iota_2^* M.  
\end{equation}
This is the morphism induced by the counit $\widetilde{\rho}_! \widetilde{\rho}^*$ for the morphism $\widetilde{\rho}: (I \times_{/J} j_1, T_{j_2}) \rightarrow (I \times_{/J} j_2, T_{j_2})$ (composition with $\rho$). 
Now observe that $p_{2,!}'$, resp. $p_{2,!}$, by assumption, can be computed on coCartesian elements just by evaluation at any element of $I \times_{/J} j_1$ resp. $I \times_{/J} j_2$.
Therefore (\ref{eqp2}) is an isomorphism.

5.\@ is obvious. 

6.\@ 
We have to show that the inclusion $\iota: (\Delta^{\circ}_{\le 3}, S_\bullet) \rightarrow (\Delta^{\circ}, S_\bullet)$ is of descent (resp.\@ of finite descent), the category $\mathcal{D}(\Delta^{\circ}_{\le 3}, S_\bullet)^{\cocart}$ being clearly just the category of classical descent data. 
By 4.\@ and 5.\@ this amounts to showing that 
$\Delta^{\circ}_{\le 3} \times_{/\Delta^{\circ}} \Delta_n$ is connected and $\pi_1(\Delta^{\circ}_{\le 3} \times_{/\Delta^{\circ}} \Delta_n)$ is trivial. 
This is well-known. Note that it is essential to take 3 terms of $\Delta^{\circ}$ here. For example $\pi_1(\Delta^{\circ}_{\le 2} \times_{/\Delta^{\circ}} \Delta_3) \cong \Z$. 
The same holds with $\Delta^{\circ}$ replaced by $\Delta$ because there is an adjunction 
\[ \xymatrix{ \Delta^{\circ}_{\le m} \times_{/\Delta^{\circ}} \Delta_n \ar@/^2pt/[r] & \ar@/^2pt/[l] \Delta_{\le m} \times_{/\Delta} \Delta_n  } \]
and the morphism
\[ (\Delta, S_\bullet)  \rightarrow (\Delta^{\circ}, S_\bullet) \]
is of descent (resp.\@ of finite descent). To prove the latter assertion, by 4.\@ it suffices to show that $\Delta^\circ \times_{/\Delta} \Delta_n$ is contractible. 
To see this, consider the projection $p: \Delta^\circ \times_{/\Delta} \Delta_n \rightarrow \cdot$. It has a section $s$ given by mapping $\cdot$ to $\id_{\Delta_n}$. 
We construct a morphism
\[ \xi: \Delta^\circ \times_{/\Delta} \Delta_n \rightarrow \Delta^\circ \times_{/\Delta} \Delta_n \]
mapping $\alpha: \Delta_k \rightarrow \Delta_n$ to $\alpha': \Delta_{k+1} \rightarrow \Delta_n$ given by 
\[ \alpha'(i) = \begin{cases} \alpha(i) & i<k, \\ n-1 & i=k, \end{cases} \]
and mapping an injective morphism $\beta: \Delta_k \rightarrow \Delta_{k'}$ such that
\[ \xymatrix{ \Delta_k \ar[rr]^{\beta}  \ar[rd] && \Delta_{k'}  \ar[ld] \\ & \Delta_n }\]
commutes to the (still injective) morphism $\beta': \Delta_{k+1} \rightarrow \Delta_{k'+1}$ given by
\[ \beta'(i) = \begin{cases} \beta(i) & i<k, \\ k' & i=k. \end{cases} \]
We have $p \circ s = \id_{\{\cdot\}}$ and here are obvious 2-morphisms
\[ s \circ p \Rightarrow \xi \qquad \id_{\Delta^\circ \times_{/\Delta} \Delta_n} \Rightarrow \xi  \]
showing that $\Delta^\circ \times_{/\Delta} \Delta_n$ is contractible, or, what matters here, that $(\Delta^\circ \times_{/\Delta} \Delta_n, S) \rightarrow (\cdot, S)$ is of descent (resp.\@ finite descent) for any $S \in \mathcal{S}$.  
\end{proof}

We need a refinement of Proposition \ref{PROPDESCENT} 3./4.\@ which also is specific to the situation of fibered categories and will not hold in any context of
cohomological descent. Call an object $j \in J$ {\bf initial} if no morphism $j' \rightarrow j$ exists with $j \not= j'$. 
\begin{LEMMA}\label{LEMMAREFINE}
Let again $\mathcal{D} \rightarrow \mathcal{S}$ be a bifibered category with choice of a full subcategory of finite objects $\mathcal{D}^{\mathrm{f}}$ as above.
Let $(\alpha, \mu): (I, S) \rightarrow (J, T)$ be a morphism of diagrams in $\mathcal{S}$. Assume that for any object $j$ there is a morphism $k \rightarrow j$ from an initial object $k$.

Then $(\alpha, \mu)$ is of descent (resp.\@ of finite descent)
if $p_j: (j \times_{/J} I, \pr_2^* S) \rightarrow (j, T_j)$ is such that $p_j^*$ is fully-faithful for any $j \in J$ and such that
$p_j$ is of descent (resp.\@ of finite descent) for any initial object $j$.
If $\alpha$ is a Grothendieck fibration then the same holds with $j \times_{/J} I$ replaced by $j \times_J I$. 
\end{LEMMA}
\begin{proof}
By the proof of Proposition \ref{PROPDESCENT}, 3. $(\alpha, \mu)^*$ is fully-faithful because all $p_j^*$ are fully-faithful. 
We show by direct construction that $(\alpha, \mu)^*$ is essentially surjective.
Let $M$ be an $(I, S)$-module and $j \in J$ an object. Choose a morphism $\alpha_j: k \rightarrow j$ such that $k$ is initial, which is the identity if $j$ is already initial. 
Define 
\[ N(j):=(p_{k,*} \iota_k^* M) \otimes_{T(k)} T(j) = (\cdot, T(\alpha_j))^*(p_{k,*} \iota_k^* M). \] 
Note that $(p_k)_*$ is an inverse to the equivalence $(p_k)^*$. 
For a morphism $\nu: j_1 \rightarrow j_2$ we must define a morphism  $N(j_1) \otimes_{T(j_1)} T(j_2) \rightarrow N(j_2)$ or, in other words
$(\cdot, T(\nu))^* N(j_1) \rightarrow N(j_2)$. 
We have the standard 2-commutative diagram
\[ \xymatrix{
(j \times_{/J} I, \iota_j^*S)  \ar[r]^-{\iota_j} \ar[d]^{p_j} \ar@{}[rd]|\Nearrow & (I, S)  \ar[d]^{(\alpha, \mu)} \\
(j,T_j) \ar[r] & (J, T)
} \]

Denote $\mu_\nu$ the morphism induced by $\nu$: $(j_2 \times_{/J} I, \pr_2^*S) \rightarrow (j_1 \times_{/J} I, \pr_2^*S)$. 
We have
\[ p_{j_1} \circ \mu_\nu  = (\cdot, T(\nu)) \circ p_{j_2} \]
and
\[ \iota_{j_1} \circ \mu_\nu = \iota_{j_2}. \]
We give the morphism 
\[ (\cdot, T(\nu))^* (\cdot, T(\alpha_{j_2}))^* (p_{k_1,*} \iota_{k_1}^* M)  \rightarrow (\cdot, T(\alpha_{j_2}))^* (p_{k_2,*} \iota_{k_2}^* M).  \]

Because of fully-faithfulness we may do so after pulling back via $p_{j_2}^*$ and hence define $p_{j_2}^*$ applied to it as the following composition
\[ \xymatrix{
p_{j_2}^* (\cdot, T(\nu))^* (\cdot, T(\alpha_{j_1}))^* (p_{k_1,*} \iota_{k_1}^* M)  \ar[d]^\sim & p_{j_2}^* (\cdot, T(\alpha_{j_2}))^* (p_{k_2,*} \iota_{k_2}^* M) \ar[d]^\sim \\
\mu_{\nu \circ \alpha_{j_1}}^* p_{k_1}^*  p_{k_1,*} \iota_{k_1}^* M   \ar[d]^\sim &  \mu_{\alpha_{j_2}}^* p_{k_2}^* p_{k_2,*} \iota_{k_2}^* M \ar[d]^\sim \\
\mu_{\nu \circ \alpha_{j_1}}^*  \iota_{k_1}^* M   \ar[d]^\sim &  \mu_{\alpha_{j_2}}^* \iota_{k_2}^* M \ar[d]^\sim \\
 \iota_{j_2}^* M  \ar@{=}[r] &   \iota_{j_2}^* M 
} \]
using that $p_{k_1}^*$ and $p_{k_1,*}$ define an equivalence.
One checks that this association is functorial. 
\end{proof}

\section{Descent for modules}

\begin{LEMMA} \label{LEMMADESCBC}
Let $R \rightarrow R'$ be a ring homomorphism. 
For a morphism $(I, S) \rightarrow (J, T)$ of diagrams of $R$-algebras the property of being of descent for arbitrary modules implies that
$(I, S \otimes_R R') \rightarrow (J, T \otimes_R R')$ is of descent for arbitrary modules.

If $R\rightarrow R'$ is finite then  for a morphism $(I, F) \rightarrow (J, G)$ of diagrams of $R$-algebras the property of being of descent for finitely generated modules  implies that $(I, S \otimes_R R') \rightarrow (J, T \otimes_R R')$ is of descent for finitely generated modules.
\end{LEMMA}
\begin{proof}
The category of $(I, S \otimes_R R')$-modules is equivalent to the category of $(I, S)$-modules with $R'$-action, i.e.\@ to the category whose objects are pairs consisting of an object $X \in \cat{$(I, S)$-mod}$ and of a homomorphism of $R$-algebras $\rho: R' \rightarrow \End_R(X)$.  
\end{proof}

\begin{LEMMA}\label{LEMMADESCFINITE}
Let $(\alpha, \mu): (I, S) \rightarrow (J, T)$ be a morphism of diagrams of rings such that $I \times_{/J} j$ is a finite diagram for all $j$. 
If $(\alpha, \mu)^*$ is faithful then $(\alpha, \mu)^*M$ finitely generated implies $M$ finitely generated.
In particular ``of descent'' implies ``of descent for finitely generated modules''.
\end{LEMMA}

\begin{proof}
This is similar to the statement that a module is finitely generated if it becomes finitely generated after a faithfully flat ring extension. 
Let $M$ be a coCartesian module over $(J, T)$ such that $(\alpha, \mu)^* M$ is finitely generated. Let $j \in J$.  For each $(i, \rho: \alpha(i) \rightarrow j) \in I \times_{/J} j$
we know that $M(\alpha(i)) \otimes_{T_{\alpha(i)}} S_i$ is a finitely generated $S_i$-module. 
Let $\{\xi_k^{i,\rho}\}_k$ be images in $M(j)$ of the (finitely many) $M(\alpha(i))$-components of those generators. 
We claim that the union over those finite sets for all objects in $I \times_{/J} j$ generates $M(j)$. For let $N(j)$ be the submodule generated by them, and
assume that $N(j)$ is different from $M(j)$. 
The non-zero morphism $j^*M \rightarrow M(j)/N(j)$ induces a non-zero morphism $M \rightarrow j_* (M(j)/N(j))$ and therefore a non-zero morphism 
$(\alpha, \mu)^* M \rightarrow (\alpha, \mu)^* j_*(M(j)/N(j))$ because $(\alpha, \mu)^*$ is faithful. For any $i$ consider the morphism
$i^* (\alpha, \mu)^* M \rightarrow i^* (\alpha, \mu)^* j_*M(j)/N(j)$ which is
\[ M(\alpha(i)) \otimes_{T(\alpha_i)} S(i)  \rightarrow (\alpha(i)^* (j_* M(j)/N(j))) \otimes_{T(\alpha_i)} S(i),  \]
or also
\[ M(\alpha(i)) \otimes_{T(\alpha_i)} S(i)  \rightarrow (\prod_{\alpha(i) \rightarrow j} M(j)/N(j))  \otimes_{T(\alpha_i)} S(i).   \]
This is the tensor product with $S(i)$ of the map induced by the canonical ones $M(\alpha(i)) \rightarrow M(j)$. 
By construction of $N(j)$ this map is zero, a contradiction. 
\end{proof}

\section{Descent for modules on ringed spaces}

\begin{LEMMA}\label{LEMMADESCLOCAL}
For the bifibered category 
\[ \cat{$X$-mod} \rightarrow \cat{$X$-ring} \]
we have that $(I, S) \rightarrow (J, T)$ is of descent (resp.\@ of finite descent) if for any open set $U \subset X$
there is a cover $U = \bigcup_i U_i$ such that $(I, S|_{U_i}) \rightarrow (J, T|_{U_i})$ is of descent (resp.\@ of finite descent) for the bifibered category
\[ \cat{$U_i$-mod} \rightarrow \cat{$U_i$-ring}. \]
\end{LEMMA}
\begin{proof}
This is an obvious glueing argument. Alternatively one could construct a commutative square of diagram-morphisms
\[ \xymatrix{
(I \times \Delta^\circ, S_\bullet) \ar[r] \ar[d] & (J \times \Delta^\circ, T_\bullet) \ar[d] \\
(I, S) \ar[r] & (J, T)
} \]
where the vertical morphisms consist point-wise in $I$ (resp.\@ $J$) of the restrictions of $S_i$ (resp. $T_j$) to a hypercovering of $X$ such
that the top horizontal morphism consists point-wise in $\Delta^\circ$ of a morphism of descent (resp.\@ of finite descent). 
By explicit construction one sees that the upper horizontal morphism is also of descent (resp.\@ of finite descent). The vertical morphisms
are then of descent by the definition of sheaf. This shows that also the lower horizontal morphism is of descent (resp.\@ of finite descent).
\end{proof}

\section{Descent and projective systems}

Let $S$ be a noetherian ring, $\mathfrak{a}$ an ideal of $S$
and consider the diagram $(\N^{\op}, S_\bullet)$ where $S_n = S/\mathfrak{a}^nS$ for every $n \in \N$.
\begin{LEMMA}
For an object  $M_\bullet \in \cat{$(\N^{\op}, S_\bullet)$-mod}$ the following assertions are equivalent
\begin{enumerate}
\item $M_\bullet \in \cat{$(\N^{\op}, S_\bullet)$-mod-f.g.-cocart}$.
\item $M_1$ is finitely generated and for each for each $k \le l$, the sequence 
\[ \xymatrix{ 0 \ar[r] & \mathfrak{a}^{k} M_l \ar[r] & M_l \ar[r] & M_k \ar[r] & 0  } \]
is exact. 
\end{enumerate}
\end{LEMMA}

\begin{proof}
The exact sequence
\[ \xymatrix{ 0 \ar[r] &  \mathfrak{a}^{k} S  \ar[r] & S \ar[r] & S/ \mathfrak{a}^k S \ar[r] & 0  } \]
tensored with $M_l$ yields the sequence
\[ \xymatrix{ 0 \ar[r] &  \mathfrak{a}^{k} M_l  \ar[r] & M_l \ar[r] & (S/ \mathfrak{a}^k S) \otimes_S M_l \ar[r] & 0  } \]
Hence coCartesianity of the diagram is equivalent to the exactness of the sequence above. 
It suffices to show that for a coCartesian diagram the condition that $M_1$  is finitely generated implies that $M_k$ is finitely generated. 
Consider the sequence of $S / \mathfrak{a}^{l} S$-modules
\[ \xymatrix{ 0 \ar[r] &  \mathfrak{a} M_l  \ar[r] & M_l \ar[r] & M_1 \ar[r] & 0  } \]
Since $\mathfrak{a}$ is nilpotent in $S / \mathfrak{a}^{l} S$, this implies that $M_l$ is finitely generated by Nakayama's lemma. 
\end{proof}

\begin{LEMMA}\label{LEMMAPROJSYS}
Let $R$ be a noetherian ring, $\mathfrak{a}$ an ideal and consider a diagram $(I, F)$ of $\mathfrak{a}$-adically complete and separated noetherian $R$-algebras. 
Let $(I \times \N^{\op}, F_\bullet)$ to be the diagram with value $F_n(i) := F(i) \otimes_R R/ \mathfrak{a}^n R$. Let $p: (I \times \N^{\op}, F_\bullet) \rightarrow (I, F)$ be the obvious morphism. 
\begin{enumerate}
\item $p^*$ and $p_*$ induce an equivalence 
\[ \cat{$(I \times \N^{\op}, F_\bullet)$-mod-f.g.-$I$-cocart} \leftrightarrow \cat{$(I, F)$-mod-f.g.}  \]
\item This equivalence restricts to an equivalence
\[ \cat{$(I \times \N^{\op}, F_\bullet)$-mod-f.g.-cocart} \leftrightarrow \cat{$(I, F)$-mod-f.g.-cocart}  \]
of the full subcategories of coCartesian modules. In other words the morphism $(I \times \N^{\op}, F_\bullet) \rightarrow (I, F)$ is of descent for finitely generated modules.
\end{enumerate}
\end{LEMMA}
\begin{proof}
1.\@ By Proposition~\ref{PROPDESCENT}, 4.\@ the statement can be checked point-wise. This reduces to the case $I = \cdot$. Then
the equivalence results from \cite[Chapitre 0, Proposition (7.2.9) and Corollaire (7.2.10)]{EGAI}.

2.\@ We will show that for a finitely generated coCartesian $(I \times \N, F_\bullet)$-module $M$, the module $p_* M$ is again coCartesian. 
Let $\nu: i \rightarrow j$ be a morphism in $I$. We have to show that the morphism
\[ (\lim_n M_n(j)) \otimes_{F(j)} F(i) \rightarrow \lim_n M_n(i) = \lim_n (M_n(j) \otimes_{F_n(j)} F_n(i)) \]
is an isomorphism. Using what is already proven, we may write this denoting $M(j):=  \lim_n M_n(j)$
\[ M(j) \otimes_{F(j)} F(i) \rightarrow \lim_n (M(j) \otimes_{F(j)} F_n(i)). \]

That this is an isomorphism follows from the following lemma. 
\end{proof}

\begin{LEMMA}
Let $R$ be a ring with ideal $\mathfrak{a}$.
Assume that $R$ is noetherian and $\mathfrak{a}$-adically complete and separated. Let $M$ be a f.g.\@ $R$-module, and let $S_n$ be a projective system of $R/\mathfrak{a}^n R$-modules (or algebras).
Then we have:
\[ M \otimes_R (\lim_n S_n) \cong \lim_n (M \otimes_{R} S_n).  \]
\end{LEMMA}
\begin{proof}
Since $R$ is noetherian, we have an exact sequence of $R$-modules
\[ \xymatrix{ R^k \ar[r] & R^m \ar[r] & M \ar[r] & 0  } \]
hence the exact sequence
\[ \xymatrix{ S_n^k \ar[r] & S_n^m \ar[r] & S_n \otimes_{R} M \ar[r] & 0  } \]
and, taking the limit, the exact sequence
\[ \xymatrix{ (\lim_n S_n)^k \ar[r] & (\lim_n S_n)^m \ar[r] & \lim_n (S_n \otimes_{R} M) \ar[r] & 0.  } \]
On the other hand by tensoring the original sequence with $\lim_n S_n$, we get
\[ \xymatrix{ (\lim_n S_n)^k \ar[r] & (\lim_n S_n)^m \ar[r] & M \otimes_R (\lim_n S_n) \ar[r] & 0.  } \]
Comparing the two exact sequences proves the assertion.
\end{proof}

\begin{PROP}
Let $R$ be a ring with ideal $\mathfrak{a}$.

Let $(\alpha, \mu): (I, F) \rightarrow (J, G)$ be a morphism of diagrams of noetherian $R$-algebras
 such that $C_{\mathfrak{a}} F$ and $C_{\mathfrak{a}} G$ consist of separated and noetherian $R$-algebras, where $C_{\mathfrak{a}}$ means completion w.r.t.\@ to the $\mathfrak{a}$-adic topology.

If $(\alpha, \mu)$ is of descent for finitely generated modules then  
$(I, C_{\mathfrak{a}} F) \rightarrow (J, C_{\mathfrak{a}} G)$ is of descent for finitely generated modules. 
\end{PROP}
\begin{proof}
Define $F_n$ as $F \otimes_R R/\mathfrak{a}^nR$. We have by definition $C_{\mathfrak{a}} F = \lim_n F_n$.
Lemma \ref{LEMMADESCBC} implies that $(I,  F_n) \rightarrow (J, G_n)$ is of descent for finitely generated modules.
Therefore by Proposition~\ref{PROPDESCENT}, 4.\@ the morphism $(I \times \N^{\op}, F_\bullet) \rightarrow (J \times \N^{\op}, G_\bullet)$ is of descent for finitely generated modules. 
Now we have the commutative diagram of diagrams of rings
\[ \xymatrix{
(I \times \N^{\op}, F_\bullet) \ar[r] \ar[d] & (J \times \N^{\op}, G_\bullet) \ar[d] \\
(I, C_\mathfrak{a}F) \ar[r] & (J,  C_\mathfrak{a}G)
} \]
in which the upper horizontal morphism is of descent for finitely generated modules and the vertical ones are of descent for finitely generated modules by Lemma~\ref{LEMMAPROJSYS}. Hence so is the lower horizontal one. 
\end{proof}

\section{Descent along basic formal/open coverings}

\begin{PROP}\label{BASICFORMALDESCENT}
Let $R$ be a noetherian ring and $f$ a non-zero divisor of $R$. Denote $\widehat{R}$ the completion of $R$ w.r.t.\@ $(f)$-adic topology and let $R_f$, and $\widehat{R}_f$, the
rings $R[f^{-1}]$, and $\widehat{R}[f^{-1}]$, respectively. 
Then 
\begin{enumerate}
\item The morphism of diagrams 
\[ p: D:=\left( \vcenter{ \xymatrix{& \widehat{R}  \ar[d]  \\ R_f  \ar[r] &  \widehat{R}_f } }  \right) \rightarrow  R \]
is of descent for arbitrary modules (resp.\@ for finitely generated modules). 

\item  For any sequence of ideals $I_i$ and elements $f_i$ such that $I_{f_i} C_{I_i} \cdots R \hookrightarrow I_{f_i} C_{I_i} \cdots R_f$ is injective, and the morphism 
\[ \widetilde{p}:  I_{f_i} C_{I_i} \cdots  D \rightarrow  I_{f_i} C_{I_i} \cdots R, \]
the functor $\widetilde{p}^*$ is fully-faithful. Here $I_{f}$, for an element $f \in R$, denotes the functor $R' \mapsto  R'[f^{-1}]$.
\end{enumerate}
\end{PROP}

\begin{proof}
We will actually need only the following axioms on the diagram of rings $D$:
\begin{enumerate}
\item[(a)] $R \rightarrow R_f$ and $R \rightarrow \widehat{R}$ are {\em flat} $R$-algebras and $\widehat{R}_f \cong \widehat{R} \otimes_R R_f$.
\item[(b)] The sequence
\begin{equation} \label{eqrf}
\xymatrix{ 0 \ar[r] & R \ar[r] & R_f \oplus \widehat{R} \ar[r] & \widehat{R}_f \ar[r] & 0  } 
\end{equation}
is exact.
\item[(c)] For each object (i.e.\@ descent datum) $(\widehat{M}, M_f, \widehat{M}_f) \in \cat{$D$-mod-cocart}$ the map
\[ \widehat{M} \oplus M_f \rightarrow \widehat{M}_f \]
is surjective. 
\end{enumerate}

Let us verify that the axioms (a--c) hold in the situation of the lemma. 
(a) Flatness of $R_f$ is clear. $\widehat{R}$ is flat, because $R$ is noetherian. The tensor property holds by construction.

(b) That the last map is surjective is clear. Hence the statement boils down to the Cartesianity of the diagram
of $R$-modules  
\[ \xymatrix{
R \ar[r]^{f^n} \ar[d] & R \ar[d] \\
\widehat{R} \ar[r]^{f^n} & \widehat{R}
} \]
for any $n$.
This diagram is Cartesian because $f$ is not a zero divisor in $R$ (and hence neither in $\widehat{R}$) and we have an isomorphism 
\[ R/f^nR \cong \widehat{R} / f^n \widehat{R}. \]
(c) Let $\rho$ be the composition of the $\widehat{R}_f$-module isomorphism $M_f \otimes_{R_f} \widehat{R}_f \rightarrow \widehat{M}_f$ and the inverse of
$\widehat{M} \otimes_{\widehat{R}} \widehat{R}_f \rightarrow \widehat{M}_f$. Let $m_f$ be any element of $M_f$. 
 We have for all $q \in \widehat{R}_f$:
\[ \rho^{-1}(m_f \otimes q) = \sum_j \widehat{m}_{j} \otimes f^{-n_{j}} p_j q , \]
where $p_j \in \widehat{R}$ are independent of $q$. Hence this element is of the form $\widehat{m} \otimes 1$ if $q$ is sufficiently divisible by $f$.
Therefore, writing any given $m_f \otimes p$, where $p \in \widehat{R}_f$ as $m_f' \otimes 1 + m_f \otimes q$ where $q$ is sufficiently divisible by $f$, we see that the map $\widehat{M} \oplus M_f \rightarrow \widehat{M}_f$ is surjective.

Now assume that the axioms (a--c) hold. 
First observe that also $\widehat{R}_f$ is flat over $R$ (base change of flatness).
Let $M$ be an $R$-module. Tensoring the sequence (\ref{eqrf}) with $M$ over $R$ yields the exact sequence
\[ \xymatrix{ 0 \ar[r] & M \ar[r] & M_f \otimes \widehat{M} \ar[r] & \widehat{M}_f \ar[r] & 0 }  \]
where $M_f:=M \otimes_R R_f$, $\widehat{M} := M \otimes_R \widehat{R}$, and $\widehat{M} \otimes_R \widehat{R}_f$.
This shows that $M$ can be reconstructed as the limit over the diagram
\[ \left( \vcenter{ \xymatrix{   & \widehat{M}  \ar[d] \\ M_f \ar[r] & \widehat{M}_f } }  \right)  \]
Consequently the unit $\id \rightarrow p_* p^*$ of the adjunction is an isomorphism. 

For the second assertion of the proposition observe that application of the functors $C_{I_i}$, and $I_{f_i}$, respectively, induce exact sequences
\[ \xymatrix{ 0 \ar[r] & I_{f_i} C_{I_i} \cdots R \ar[r] & I_{f_i} C_{I_i} \cdots R_f \oplus  I_{f_i} C_{I_i} \cdots \widehat{R} \ar[r] & I_{f_i} C_{I_i} \cdots \widehat{R}_f \ar[r] & 0  } \] 
Therefore also after applying those functors to $D$, and $R$, respectively, we have that the unit $\id \rightarrow \widetilde{p}_* \widetilde{p}^*$ of the new adjunction is an isomorphism.

Now let $(\widehat{M}, M_f, \widehat{M}_f)$ be a descent datum. We form the exact sequence (cf.\@ axiom (c) for the surjectivety of the map to $\widehat{M}_f$):
\[ \xymatrix{ 0 \ar[r] & N \ar[r] & M_f \otimes \widehat{M} \ar[r] & \widehat{M}_f \ar[r] & 0 }  \]
To see that the counit $p^* p_* \rightarrow \id$ is an isomorphism, we have to show that the natural maps $N \otimes_R R_f \rightarrow M_f$ and $N \otimes_R \widehat{R} \rightarrow \widehat{M}$ are isomorphisms. 
Exactness of the sequence (\ref{eqrf}) yields that the following diagram is Cartesian
\[ \xymatrix{
R \ar[r] \ar[d] & R_f \ar[d] \\
\widehat{R}  \ar[r] & \widehat{R}_f
} \]
From the Cartesianity, we may conclude that in the following diagram the right vertical morphisms is an isomorphism:
\[ \xymatrix{
0 \ar[r] & R \ar[r] \ar[d] & R_f \ar[d] \ar[r] & \ar[d]^\sim R_f/R\ar[r] & 0   \\
0 \ar[r] & \widehat{R} \ar[r] & \widehat{R}_f  \ar[r] & \widehat{R}_f/ \widehat{R}\ar[r] & 0
} \]
Tensoring this diagram with $\widehat{R}$ over $R$ observe that the whole morphism of exact sequences splits (via the multiplication maps), hence we get a diagram with exact rows and columns:
\[ \xymatrix{
& 0 \ar[d] & 0 \ar[d]  \\ 
& \mathrm{ker}  \ar[d] \ar@{=}[r] & \mathrm{ker} \ar[d]  \\ 
0 \ar[r] & \widehat{R} \otimes_R \widehat{R}  \ar[d] \ar[r] & \widehat{R}_f \otimes_R \widehat{R}  \ar[d] \ar[r] & (\widehat{R}_f/ \widehat{R}) \otimes_R \widehat{R} \ar[r]  \ar[d]^\sim & 0 \\
0 \ar[r] & \widehat{R} \ar[d] \ar[r] & \widehat{R}_f \ar[r] \ar[d] & (R_f/R) \otimes \widehat{R} \ar[r] & 0   \\
& 0 & 0 
} \]
Tensoring the diagram over $\widehat{R}$ with $\widehat{M}$ we get exact colums which we may insert in the following diagram
which gets exact rows and columns
\[ \xymatrix{
&  & 0 \ar[d] & 0 \ar[d]  \\ 
& & \mathrm{ker} \otimes_{\widehat{R}} \widehat{M}  \ar[d] \ar@{=}[r] &  \mathrm{ker} \otimes_{\widehat{R}} \widehat{M}  \ar[d]  \\ 
0 \ar[r] & N \otimes_R \widehat{R}  \ar[d] \ar[r] & \widehat{M} \otimes_R \widehat{R}  \oplus \widehat{M}_f \ar[d] \ar[r] & \widehat{M}_f \otimes_R \widehat{R}  \ar[r]  \ar[d] & 0 \\
0 \ar[r] & \widehat{M} \ar[r] & \widehat{M} \oplus \widehat{M}_f \ar[r] \ar[d] & \widehat{M}_f \ar[r] \ar[d] & 0  \\
 & & 0 & 0 
} \]
This shows that the natural map $N \otimes_R \widehat{R} \rightarrow  \widehat{M}$ is an isomorphism. 

That also $N \otimes_{R} R_f \rightarrow M_f$ is an isomorphism follows because the axioms (a--c) are completely symmetric in $\widehat{R}$ and $R_f$. 
Actually, if we have, as in the formulation of the proposition, that $R \rightarrow R_f$ is an epimorphism of $R$-algebras, i.e.\@ that $R_f = R_f \otimes_R R_f$, then
this is even easier. 

Finally the statement of descent for {\em finitely generated} modules follows from Lemma~\ref{LEMMADESCFINITE}.
\end{proof}

\section{Descent along completions w.r.t.\@\ a divisor with normal crossings}

\begin{DEF}\label{DEFDSNC}
A subscheme $D$ on a regular scheme $S$ is called a {\bf divisor with strict normal crossings} if it is the zero-locus of a Cartier divisor which is Zariski locally around any $p \in D$ of
the form $f_1 \cdots f_m$, where $f_1, \dots, f_m$ are part of a sequence of minimal generators of the maximal ideal at $p$. 
\end{DEF}

This definition differs slightly from \cite[Expos\'e I, 3.1.5, p.\@ 24]{SGA5} to the extent that there the existence of $f_1, \dots, f_m$ is claimed to exist globally.
Note that it follows from the definition that all $V(f_i)$ (defined locally) are regular themselves. 

\begin{PAR}
Let $X$ be a regular noetherian scheme and let $D \subset X$ be a divisor with strict normal crossings. Let $\{Y\}$ be the coarsest stratification of $X$ into locally closed subvarieties such that every component of $D$ is the closure of a stratum $Y$. For each stratum we define a sheaf of $\OO_X$-algebras on $X$:
\[ R_Y(U) := C_{\overline{Y}} \OO_X(U') =  \lim_n \OO_X(U')/\mathcal{I}_{\overline{Y}}(U')^n   \]
where $U' \subset U$ is such that $U' \cap \overline{Y} = U \cap Y$. It is clear that the definition of $R_Y$ does not depend on the choice of $U'$ for $R_Y$ may also be described
as $\iota_* \OO_{C_{\overline{Y}}X|_Y}$, where  $C_{\overline{Y}}X$ is the formal completion of $X$ along $\overline{Y}$, and $C_{\overline{Y}}X|_Y$ is the
open formal subscheme with underlying topological space equal to $Y$ and $\iota$ is the composed morphism of formal schemes $C_{\overline{Y}}X|_Y \rightarrow X$. 

For any chain of disjoint strata $Y_1, Y_2, \dots, Y_n$ such that $Y_i \subset \overline{Y_{i-1}}$, we define inductively a sheaf $R_{Y_1, \dots, Y_n}$ of $\OO_X$-algebras on $X$
which coincides with the previous one for $n=1$. For $n>1$ we set
\begin{equation}\label{constructr}
 R_{Y_1, \dots, Y_n}(U) := C_{\overline{Y_1}} \left(R_{Y_2, \dots, Y_n}(U) \otimes_{\OO_X(U)} \OO_X(U') \right).    
 \end{equation}
Again $U' \subset U$ is such that $U' \cap \overline{Y_1} = U \cap Y_1$, and this definition does not depend on the choice of $U'$. 
\end{PAR}

\begin{BEISPIEL}
Consider $X=\A^1_k$ and $D=\{0\}$. Then the strata are $D$ and $Y := X \setminus D$. We have
\[ R_Y(X) = k[x,x^{-1}] \quad R_D(X) = k\llbracket x \rrbracket \quad R_{Y,D}(X) = k\llbracket x \rrbracket[x^{-1}].  \]
\end{BEISPIEL}

Recall that for a sheaf of $\OO_X$-algebras $R$ as above, a {\bf coherent sheaf $M$ of $R$-modules} is a sheaf of $R$-modules such that on a covering $\{U_i\}$ of $X$, we have an
exact sequence
\[ R|_U^{n_i} \rightarrow R|_U^{m_i} \rightarrow M|_{U_i} \rightarrow 0 \]
for every $i$.

\begin{LEMMA}\label{LEMMARLOCAL}
If $U \subset X$ is affine then for any of the sheaves $R=R_{Y_1, \dots, Y_n}$ of $\OO_X$-modules defined above, we have an equivalence of categories
of coherent sheaves of $R|_U$-modules and finitely generated $R(U)$-modules. 
\end{LEMMA}
\begin{proof}
This is shown as the analogous property for coherent sheaves of $\OO_{\mathcal{X}}$-modules on a noetherian formal scheme $\mathcal{X}$.
\end{proof}

Coming back to the stratification $\{Y\}$ of $X$, we define the following semi-simplicial set $S$. The set $S_n$ consists of
chains of {\em disjoint} strata $[Y_1, \dots, Y_n]$ such that $Y_i \subset \overline{Y_{i-1}}$ (we write also $Y_i < Y_{i-1}$) for $i=2\dots n$ with the obvious face maps. 
Alternatively consider the stratification as a partially ordered set where $Z \le Y$ if $Z \subset \overline{Y}$. $S$ is then the 
(semi-simplicial) nerve of this partially ordered set. 

We define also the category $\int S$ whose objects are
pairs $(n, \xi)$, where $n \in \N$ and $\xi \in S_n$, and whose morphisms $\mu: (n, \xi) \rightarrow (m, \xi')$ are morphisms $\mu: \Delta_n \rightarrow \Delta_m$ in $\Delta^{\circ}$ such that
 $S(\mu) (\xi') = \xi$. This is the fibered category associated with the functor $(\Delta^\circ)^{\op} \rightarrow \cat{set} \subset \cat{cat}$ via the Grothendieck construction. 
 
 The construction (\ref{constructr}) actually defines a functor
 \begin{eqnarray*} 
 R: \int S &\rightarrow& \cat{$\OO_X$-alg} \\ 
  (n, Y_1< \dots< Y_n) &\mapsto& R_{Y_1, \dots, Y_n}. 
 \end{eqnarray*}
 
\begin{LEMMA}
The category 
\[ \cat{$(\int S, R)$-coh-cocart}  \]
is equivalent to the category of the following descent data:
For each stratum $Y$ a coherent sheaf $M_Y$ of $R_Y$-modules together with isomorphisms
\[  \rho_{Y',Z}: M_Y \otimes_{R_Y} R_{Y,Z} \rightarrow M_Z \otimes_{R_Z} R_{Y,Z} \]
 for any $Y, Z$ with $Z \subset \overline{Y}$, 
which are compatible w.r.t.\@ any triple $Y, Z, W$ of strata with $Z \subset \overline{Y}$ and $W \subset \overline{Z}$ in the obvious way.
\end{LEMMA}
 \begin{proof}
This follows essentially from Proposition~\ref{PROPDESCENT}, 6.
For this observe that the morphism $(\int S, R) \rightarrow (\Delta^\circ, R')$, where $R'(\Delta_m) := \prod_{\xi \in S_m} R(m, \xi)$
is of finite descent by Proposition~\ref{PROPDESCENT}, 3. and the fact that $S_m$ is finite.
\end{proof}

\begin{HAUPTSATZ}\label{MAINTHEOREM}
The morphism $(\int S, R) \rightarrow (\cdot, \OO_X)$ is of descent for coherent sheaves. In other words the natural functor
\[ \cat{$(\int S, R)$-coh-cocart} \rightarrow \cat{$\OO_X$-coh} \]
is an equivalence of categories. 
\end{HAUPTSATZ}
\begin{proof}
By Lemma~\ref{LEMMADESCLOCAL} and Lemma~\ref{LEMMARLOCAL} we are reduced to a local, affine situation of the following kind: 
We may assume that $D$ is defined by an equation $f_1 \cdots f_n = 0$ such that $f_1, \dots, f_n$ are part of a minimal sequence of generators
of the maximal ideal of a point $p \in D$. We may also assume (by possibly shrinking the affine cover) that all $V(f_{i_1}, \dots, f_{i_k})$ are irreducible for each subset $\{ i_1, \dots, i_n \} \subset \{1, \dots, n\}$. 
The strata $Y$ are then all of the form $V(f_{i_1}, \dots, f_{i_k}) \setminus V(f_{i_{k+1}}, \dots, f_{i_n})$ such that $\{ i_1, \dots, i_n \} = \{1, \dots, n\}$.
Denote $R:=\OO_X(U)$ for  such an open affine subset $U$. 
Then 
\[ R_Y = \lim_l R[f_{i_{k+1}}^{-1}, \dots, f_{i_n}^{-1}] / (f_{i_1}^l, \dots, f_{i_k}^l) \]
and inductively 
\[ R_{Y_1, Y_2, \dots, Y_n} = \lim_l R_{Y_2, \dots, Y_n}[f_{i_{k+1}}^{-1}, \dots, f_{i_n}^{-1}] / (f_{i_1}^l, \dots, f_{i_k}^l) \]
for $Y_1 = V(f_{i_1}, \dots, f_{i_k}) \setminus V(f_{i_{k+1}}, \dots, f_{i_n})$. 
We similarly get a diagram $(\int S, R)$ in $\Dia^{\op}(\cat{ring})$ and have to show that  
\begin{equation}\label{eqeq}
 \cat{$(\int S, R)$-mod-f.g.-cocart}  \cong \cat{$R$-mod-f.g.} 
\end{equation}
is an equivalence. Note that our semi-simplicial set $S$ here is equal to the (semi-simplicial) nerve of the partially ordered set $\mathcal{P}(\{1, \dots, n\})$ (power set). 

We show by induction on $n$ that (\ref{eqeq}) is an equivalence as follows. Denote the diagram defined before by $(\int S^{(n)}, R^{(n)})$ where $D$ has equation $f_1 \cdots f_n$.
Decreasing $n$ by 1 means forgetting the last element $f_n$. 
For $n=1$, $S^{(1)}$ is the nerve of the partially ordered set $\{\} < \{1\}$ hence $\int S^{(1)}$ is the diagram $\lefthalfcap$ and we are precisely in the situation of Proposition~\ref{BASICFORMALDESCENT}, 1. Therefore $(\int S^{(1)}, R^{(1)}) \rightarrow (\cdot, R)$ is of descent. 
For $n>1$ consider the stratification $\{Y\}$ for $f_1, \dots, f_{n-1}$. Then cut each stratum $Y$ in the two pieces $Y' = Y \cap V(f_n)$ and $Y'' = Y \setminus Y'$.
This yields a chain of stratifications and refinements like considered in the following Lemma. 
\end{proof}

\begin{LEMMA}
We consider arbitrary stratifications of $U$ such that
each stratum is locally closed and a union of the ones considered before. They are all of the form 
$V(f_{i_1}, \dots, f_{i_k}) \setminus V(f_{i_{k+1}}, \dots, f_{i_l})$, where however not necessarily $\{ i_1, \dots, i_l \} = \{1, \dots, n\}$.
The members of such a stratification form a partially ordered set $\mathcal{P}$ as before and we denote by $S$ the corresponding (semi-simplicial) nerve.

Let $\nu: \mathcal{P}' \rightarrow \mathcal{P}$ be a morphism of such partially ordered sets which comes from
a refinement of stratifications such that precisely two strata $Y', Y''$ get mapped to one
stratum $Y$ and we have $Y' = Y \cap V(f)$ and $Y'' = Y \setminus Y'$. It induces a map of semi-simplicial nerves, denoted $\nu$ by abuse of notation, as follows: 
$(\Delta_n, Y_1 < \dots < Y_n)$ is mapped to $(\Delta_{n'}, \nu(Y_1) < \dots < \nu(Y_{n}))$ where in the sequence $\nu(Y_1) < \dots < \nu(Y_n)$ double entries have been deleted such that every entry appears only once. 

Let $R$, resp.\@ $R'$ be the corresponding ring-valued functors. 

Then the functor induced by $\nu$
\[ \cat{$(\int S', R')$-mod-f.g.-cocart}  \cong \cat{$(\int S, R)$-mod-f.g.-cocart} \]
is an equivalence. 
\end{LEMMA}
\begin{proof}
{\bf Claim:} $\nu: \int S' \rightarrow \int S$ is a Grothendieck fibration.

{\em Proof of the claim: } For the fiber of $\nu$ over a $\xi=(\Delta_n, Y_1 < \dots < Y_n)$ there are two possibilities:

\begin{enumerate}
\item If $Y$ does not occur in the list, it consists of $\xi$ itself, considered as an element of $\int S'$.
\item If $Y$ occurs in the list, the fiber consists of the diagram
\[ \xymatrix{
 & (\Delta_n, Y_1 < \dots < Y' < \dots < Y_n)  \ar[d] \\
(\Delta_n, Y_1 < \dots < Y'' < \dots < Y_n)  \ar[r] &  (\Delta_{n+1}, Y_1 < \dots < Y' < Y'' < \dots < Y_n)
} \]
\end{enumerate}

For a morphism in $\int S'$, say $(\Delta_{n'}, Y_{i_1}< \dots< Y_{i_{n'}}) \rightarrow (\Delta_n, Y_1< \dots< Y_n)$, there is an obvious pull-back functor between these fibers
which establishes $\nu$ as a Grothendieck fibration.

Using the claim and Proposition~\ref{PROPDESCENT}, 3.\@ we would have to show that the fibers of $(S'_\xi , R'|_\xi) \rightarrow (\xi, R(\xi))$ are of descent for any $\xi = (\Delta_k, Y_1<\cdots<Y_k) \in \int S'$. Actually we have the refinement
Lemma~\ref{LEMMAREFINE} which reduces us to prove that for all $Z$ the fiber above $(\Delta_0, Z)$ (these are the initial objects of $\int S$ in the sense of Lemma~\ref{LEMMAREFINE}) is of descent, and that for all other $p_\xi: (S'_\xi , R'_\xi) \rightarrow (\xi, R(\xi))$ the pull-back $p_\xi^*$ is fully faithful.
In other words, we have to see that 

1. (only $Z=Y$ is non-trivial)
\[ p_Y: \left( \vcenter{ \xymatrix{
 & R_{Y'}  \ar[d] \\
R_{Y''} \ar[r] & R_{Y',Y''} 
} } \right) \rightarrow (R_Y) \]
is of  descent for finitely generated modules. This is Proposition~\ref{BASICFORMALDESCENT}, 1.

2. 
\[ p_\xi: \left( \vcenter{ \xymatrix{
 & R_{Y_1,\dots,Y',\dots,Y_n}  \ar[d] \\
R_{Y_1,\dots,Y'',\dots,Y_n} \ar[r] & R_{Y_1,\dots,Y',Y'',\dots,Y_n} 
} } \right) \rightarrow (R_{Y_1,\dots,Y,\dots,Y_n}) \]
is such that $p_\xi^*$ is fully-faithful (on finitely generated modules). This is Proposition~\ref{BASICFORMALDESCENT}, 2.
\end{proof}

\newpage

\bibliographystyle{abbrvnat}
\bibliography{descente}

\end{document}